\documentclass[reqno]{amsart}
\usepackage[latin1]{inputenc}
\usepackage{amssymb}
\usepackage{amsfonts}
\usepackage{enumerate}
\usepackage{enumitem}
\usepackage{indentfirst}
\usepackage{amsmath}
\usepackage{amsthm}
\usepackage[all]{xy}
\usepackage{mathrsfs}
\usepackage{standalone}
\usepackage{graphicx,psfrag,subfigure}
\usepackage{pst-all}
\usepackage{pifont}
\usepackage[bookmarks]{hyperref}
\usepackage{marvosym}
\usepackage{pifont}
\usepackage{amsfonts}
\usepackage{wasysym}

\def\R{\mathbb{R}}
\def\Z{\mathbb{Z}}
\def\T{\mathbb{T}}

\newtheorem*{mainthm}{Main Theorem}
\newtheorem{defi}[subsection]{Definition}

\newtheorem{prop}[subsection]{Proposition}
\newtheorem{lemma}[subsection]{Lemma}
\newtheorem{cor}[subsection]{Corollary}
\newtheorem{rmk}[subsection]{Remark}

\numberwithin{equation}{section}

\begin{document}

\title[Transitive endomorphisms with critical points]{Transitive endomorphisms with critical points}
\author{Wagner Ranter}
\date{\today}

\maketitle

\begin{abstract}
We show that a non-wandering endomorphism on the torus with topological degree at least two, hyperbolic linear part,
and for which the critical points are in some sense ``generic'' is transitive. This is an improvement of a result by
Andersson \cite{Andersson}, since it allows critical points and relaxes the volume preserving hypothesis.
\end{abstract}


\section{Introduction}

The interplay between the dynamics on the homology group and properties of dynamical
systems have attracted recently a lot of attention. One of the most well known problems in this direction is
the Entropy conjecture of Shub (see \cite{Shub}). In a sense, one tries to obtain some dynamical properties
(which are of asymptotic nature) by the a priori knowledge of how a certain map wraps the manifold in itself.

In this paper we are interested how a dynamical systems could be influenced by the its action
on the homology groups. In particular, we are interested on conditions on the action on the homology group of
a continuous map of the torus that allow to promote a mild recurrence property (being non-wandering) to a stronger one (i.e., transitivity).
This improves a recent result of Andersson (see \cite{Andersson}) by allowing the presence of critical points.

Let us fix some notations. Let $\T^2$ be two-dimensional torus and let $M_2(\Z)$ be the set of all square
matrices with integer entries. A toral endomorphism or, simply, endomorphism is a surjective continuous map $f:\T^2 \to \T^2$.
It is well known that given two endomorphisms $f, g:\T^2 \to \T^2$ then $f$ and $g$ are homotopic if and only if
$f_{\ast}=g_{\ast}:H_1(\T^2)\to H_1(\T^2)$. From this fact, we have that given a continuous map $f:\T^2 \to \T^2$
there is a unique square matrix $L \in M_2(\Z)$ such that the linear endomorphism induced by $L$,
denoted by $L:\T^2 \to \T^2$ as well, is homotopic to $f$. The matrix $L$, we call linear part of $f$.
When $L$ is a hyperbolic matrix\footnote{The matrix has no eigenvalues of modulus one.}, it called hyperbolic linear part. 

Let $f:\T^2 \to \T^2$ be an endomorphism with linear part $L \in M_2(\Z)$.
We define the {\it{topological degree of $f$}} as the determinant of $L$.

The following question naturally arises:
\bigskip

{\bf{Question 1:}} Under which conditions an endomorphism with hyperbolic linear part is transitive?\\

Recently, Andersson (in \cite{Andersson}) showed that volume preserving non-invertible covering maps of the torus
with hyperbolic linear part is transitive.

It is interesting to observe that the hyperbolicity property is on the homology group, and this action
influences the dynamics. Note that, when the action is no hyperbolic this result can not true. For instance,
$f:\T^2 \to \T^2, \, f(x,y)=(x,3y),$ is a volume preserving non-invertible covering map which preserves vertical stripes.

In order to study dynamical systems in the $C^0$-topology, it is interesting to consider continuous maps
with critical points\footnote{The points which locally the map is not a homeomorphism},
because the set of all the covering maps (endomorphisms without critical points) is not neither dense, nor open set.
Then, questions naturally appear about the critical set. For instance:

\bigskip
{\bf{Question 2:}} Can the result be extended to allow critical points?\\

Another questions that can be done is the following:

\bigskip
{\bf{Question 3:}} Can the volume preserving condition be relaxed?\\

In this direction, we are interested to give some answer about questions 2 and 3. We show that it is possible to obtain an analogous
result changing the volume preserving property given by a milder topological property even in the case where there are
critical points. Notice that one can create sinks for maps of $\T^2$ in any homotopy class, so at least some sort of a priori
recurrence is necessary to obtain such result.

In order to state the main result of this work, let us introduce some notations before.

A point $p \in \T^2$ is a {\it{non-wandering point for $f$}} if for every neighborhood
$B_p$ of $p$ in $\T^2$ there exists an integer $n \geq 1$ such that $f^n(B_p)\cap B_p$ is nonempty. The set $\Omega(f)$ of all
non-wandering points is called non-wandering set. Clearly $\Omega(f)$ is closed and $f$-forward invariant. We call an endomorphism
$f:\T^2 \to \T^2$ by non-wandering endomorphism if $\Omega(f)=\T^2$. Recalling, a point $p$ belonging to $\T^2$ is said
to be a {\it{critical point}} for $f$ if for every neighborhood $B_p$ of $p$ in $\T^2$, we have that $f:B_x \to f(B_x)$ is not a homeomorphism.
We will denote by $S_f$ the set of all the critical points. Clearly $S_f$ is a closed set in $\T^2$. A critical point $p$ is called
{\it{generic critical point}} if for any neighborhood $B$ of $p$ in $\T^2$, $f(B)\backslash \{f(p)\}$ is a connected set.
When all critical points are generics, $S_f$ we will be called generic critical set. It is easy to see that the fold and cusp
critical points are {\emph{generic critical points}}, this justifies the name since by H. Whitney (see \cite{Whitney}) the maps whose
critical points are folds and cusps are generic in the $C^{\infty}$-topology.

\medskip
In this paper, we will prove the following result:
\medskip

\begin{mainthm}\label{thm-1}
Let $f:\T^2 \to \T^2$ be a non-wandering endomorphism with topological degree at least two and generic critical set.
If $f$ is not transitive, then its linear part has a real eigenvalue of modulus one.
\end{mainthm}

It is not known whether the hypothesis of generic critical set is a necessary condition.
It is utilized as a technical hypothesis.

The paper is organized as follows. In section \ref{section-2}, we present a rephrased of the main theorem and some corollaries.
After, in section \ref{section-3}, we give a sketch of the proof of the main theorem. In sections \ref{section-4} and \ref{section-5},
we prove some results that will be used in the proof of the main theorem.

\section{The results}\label{section-2}

In this section we will give the main theorem and its consequences. The a main theorem can be rephrased as follows:

\begin{mainthm}\label{thm 1}
Let $f:\T^2 \to \T^2$ be a non-wandering endomorphism with topological degree at least two and generic critical set.
If linear part of $f$ is hyperbolic, then $f$ is transitive.
\end{mainthm}

Before starting of the proof, we give some immediate consequences of the main theorem:

\begin{cor}\label{cor 1}
Let $f:\T^2 \to \T^2$ be a volume preserving endomorphism with topological degree at least two and generic critical set.
If $f$ is not transitive, then its linear part has a real eigenvalue of modulus one. 
\end{cor}

The proof follows from of the fact that volume preserving implies that the non-wandering set is the whole torus.
Furthermore, in the case that the critical set is empty. That is, when the endomorphism is a covering maps.
We also have the following consequence:

\begin{cor}\label{cor 1.1}
Let $f:\T^2 \to \T^2$ be a non-wandering endomorphism with topological degree at least two and without critical points (i.e., $S_f=\emptyset$).
If $f$ is not transitive, then its linear part has a real eigenvalue of modulus one.
\end{cor}

\section{Sketch of the proof of the main theorem}\label{section-3}

We prove in the section \ref{section-4} that if a non-wandering endomorphism is not transitive, then we can divide the torus in two complementary
open sets which are $f$-invariant. After, in the section \ref{section-5}, we use the generic critical points to prove that those open sets are
essential (see Definition \ref{def-1}) and their fundamental groups have just one generator. 
Then, in section \ref{section-6}, we prove that the action of $f$ on the fundamental group of the torus
has integer eigenvalues and that at least one has modulus one. 

\section{Existence of invariant sets}\label{section-4}

An open subset $U\subset \T^2$ is called {\it{regular}} if $U=\rm{int}(\overline{U})$ where $\overline{U}$ is the closure of $U$ in $\T^2$
that sometimes we will also be denoted like ${\bf{cl}}(U)$.

Given a subset $A\subset \T^2$ we write $A^{\bot}:=\T^2 \backslash \overline{A}$.
Note that for any open set $U\subseteq \T^2$, we have $U^{\bot}=\rm{int}(\overline{U^{\bot}})$, i. e., $U^{\bot}$ is regular.

We say that a subset $A\subseteq \T^2$ is {\it{f-backward invariant}} if $f^{-1}(A)\subseteq A$ and {\it{f-forward invariant}}
if $f(A)\subseteq A$. We say that $A\subseteq \T^2$ is {\it{f-invariant}}
when it is $f$-backward and $f$-forward invariant set.

An endomorphism $f:\T^2 \to \T^2$ is transitive if for every open set $U$ in $\T^2$
we have that $\cup_{n \geq 0}f^{-n}(U)$ is dense in $\T^2$.

The lemma below gives a topological obstruction for a non-wandering endomorphism to be transitive.

\begin{lemma}\label{lema 1}
Let $f:\T^2 \to \T^2$ be a non-wandering endomorphism. Then, the following are equivalent:

\begin{enumerate}

\item[$(a)$] $f$ is not transitive;

\item[$(b)$] there exist $U,V \subseteq \T^2$ disjoint $f$-backward invariant regular open sets. Furthermore, 
$\overline{U}$ and $\overline{V}$ are $f$-forward invariant.

\end{enumerate}

\end{lemma}

\begin{proof}
$(b)\Rightarrow (a)$: It is clear. Since $f^{-n}(U)\cap V=\emptyset$ for every $n\geq 0$.\bigskip

$(a)\Rightarrow (b)$: Since $f$ is not transitive, there exist $U_0'$ and $V_0'$ open sets such that
$$f^{-n}(U_0')\cap V_0'=\emptyset \ \ \text{for every} \ \ n\geq 0.$$

\medskip
{{\it{Claim}} 1:} $U'=\cup_{n\geq 0}f^{-n}(U_0')$ and $V'=\cup_{n\geq 0}f^{-n}(V_0')$
are disjoint $f$-backward invariant open sets.
\bigskip

Indeed, it is clear that $U'$ and $V'$ are $f$-backward invariant open sets. Then, we must show only that
$U'$ and $V'$ are disjoint sets. For this, suppose by contradiction that $U'\cap V'\neq \emptyset$.
That is, suppose that there exist $n, m \geq 0$ such that $$f^{-n}(U_0')\cap f^{-m}(V_0')\not=\emptyset.$$

Let $x \in f^{-n}(U_0')\cap f^{-m}(V_0')$. Then $f^n(x) \in U_0'$ and $f^m(x)\in V_0'$. 

Then, we have the following possibilities:

\begin{itemize}
\item $n \geq m:\ \ f^{n-m}(f^m(x)) \in U_0' \Rightarrow f^{-n+m}(U_0')\cap V_0'\not=\emptyset.$\\
\item $n < m:$ By continuity of $f$, we can take a neighborhood $B\subseteq U_0'$ of $f^n(x)$ such that
$f^{m-n}(B) \subseteq V_0'$. Since $\Omega(f)=\T^2$, we can take $B$ and $k\geq m-n$ such that
$f^k(B)\cap B\neq \emptyset$. Hence, $f^{-(k-m+n)}(U_0')\cap V_0' \neq \emptyset.$
\end{itemize}

In both cases, we have a contradiction.\bigskip

The following statement will be used to choose the sets $U$ and $V$.

\medskip
{{\it{Claim}} 2:} $f^{-1}(U')$ is dense in $U'$. The same holds for $V'$.
\bigskip

Indeed, given any open subset $B$ of $\T^2$ contained in $U'$, since $f$ is a non-wandering endomorphism there exists
$n\geq 1$ such that $f^n(B)\cap B\not= \emptyset$. Then, $f^{-n}(B)\cap B\not=\emptyset$, in particular,
$f^{-n}(U')\cap B\not=\emptyset$. Therefore $f^{-1}(U')\cap B\not=\emptyset$ , since $f^{-m}(U')\subseteq f^{-1}(U')$ for
all $m\geq 1$. In particular, $\overline{U'}=\overline{f^{-1}(U')}$. This proves the claim 2.\\


Finally, we define

\begin{equation}\label{def-U}
 U={\rm{int}}(\overline{U'}) \ \ {\rm{and}} \ \ V={\rm{int}}(\overline{V'}).
\end{equation}

\medskip
{{\it{Claim}} 3:} $U$ and $V$ satisfy:
\medskip
\begin{enumerate}
\item[($i$)] $U$ and $V$ are regular;
\item[($ii$)] $f^{-1}(U)\subseteq U$ and $f^{-1}(\overline{U})\supseteq \overline{U}$, the same holds for $V$.\\
\end{enumerate}

Item ($i$) follows from the fact that $\overline{U}=\overline{U'}$.
\medskip

To prove item ($ii$), it is sufficient to show that 
$${\rm{int}}(f^{-1}(\overline{U'}))=U.$$
Because $f^{-1}(U)\subseteq {\rm{int}}(f^{-1}(\overline{U'}))$, since
$f^{-1}(U)=f^{-1}({\rm{int}}(\overline{U'}))\subseteq f^{-1}(\overline{U'})$. Hence, we have $f^{-1}(U)\subseteq U$
and $\overline{U}=\overline{f^{-1}(U)}\subseteq f^{-1}(\overline{U})$, by Claim 2.

Now, we will prove that 

\begin{equation*}\label{eq.1}\tag{$\ast$}
{\rm{int}}(f^{-1}(\overline{U'}))=U.
\end{equation*}

Note that $U={\rm{int}}(\overline{U'})\subseteq {\rm{int}}(f^{-1}(\overline{U'}))$, since
$\overline{U'} = \overline{f^{-1}(U')} \subseteq f^{-1}(\overline{U'})$. Hence, we have to show only that
\begin{equation*}\label{eq.2}\tag{$\ast\ast$}
{\rm{int}}(f^{-1}(\overline{U'}))\subseteq U.
\end{equation*}

To prove this, let $B$ be an open set contained in $f^{-1}(\overline{U'})$. Suppose that $B$ is not contained in $\overline{U'}$.
Then, we may take an open subset $B'$ of $\T^2$ contained in $B$ such that $B'\cap \overline{U'}=\emptyset$. Since $\Omega(f)=\T^2$
and $f^n(B')\subseteq f^n(\overline{U'}) \subseteq \overline{U'}$ for every $n \geq 1$, we have a contradiction because
$f^n(B')\cap B'\neq \emptyset$ for $n\geq 1$. Therefore, $B$ is contained in $\overline{U'}$.
Thus, we conclude (\ref{eq.2}), and so, (\ref{eq.1}). This proves the Claim 3.
\end{proof}

Henceforth, we assume that $f$ is a non-wandering endomorphism with topological degree at least two
and $U,V$ are the sets given by proof of the item ($b$) of the lemma above.

\begin{rmk}\label{rmk 1}
 Note that as $f^{-1}(U)\subseteq U$ and $f^{-1}(\overline{U})\supset \overline{U}$, one gets $f(\overline{U})= \overline{U}$
 and, consequently, $f(\partial \overline{U})=\partial \overline{U}$. Moreover, since ${\rm{int}}(\overline{U})=U$, 
 $\partial U= \partial \overline{U}$, one has \linebreak $\partial U_i \subseteq \partial U$
 for every $U_i$ connected component of $U$. Thus, given $U_i$ a connected component of $U$, we have
 $f(\partial U_i)\subseteq \partial U$.
\end{rmk}
%
%
%
%
%
%
%
%

The following proposition shows that the points belonging to $U$ whose images are in the boundary of $U$ are critical points.  

\begin{prop}
Let $p \in U$. If $f(p)\in \partial U$ then $p \in S_f$.
\end{prop}

\begin{proof}
Suppose that there exist a neighborhood $B$ of $p$ contained in $U$ such that $f:B\to f(B)$ is a homeomorphism and $f(B)$ is an open set
contained in $\overline{U}$. In particular, $f(B)\subseteq {\rm{int}}(\overline{U})=U$.
\end{proof}

The following lemma shows that the image of a component of $U$ which intersect two other components of $U$ intersects the boundary of $U$
in a unique point.

\begin{lemma}\label{lema 2}
Given $U_0,\,U_1$ and $U_2$ connected components of $U$ such that $U_1$ and $U_2$ are disjoint and
let $U_{01}$ and $U_{02}$ be connected components of $f^{-1}(U_1),\, f^{-1}(U_2)$ contained in $U_0$, respectively.
If $C:=\partial U_{01}\cap \partial U_{02}$ is a non-empty set contained in $U_0$, then $f(C)$ is a point.
\end{lemma}

\begin{proof}
Consider $C':=f(C)$, without loss of generality, suppose that $C$ is a nontrivial connected set. Then, as
$f(\partial U_i)\subseteq \partial U$, we have $C'\subseteq \partial U_1 \cap \partial U_2$ is a connected set.
\newline
\begin{figure}[!h]\label{finitascomponentes-0}
\centering
\includegraphics[scale=0.8]{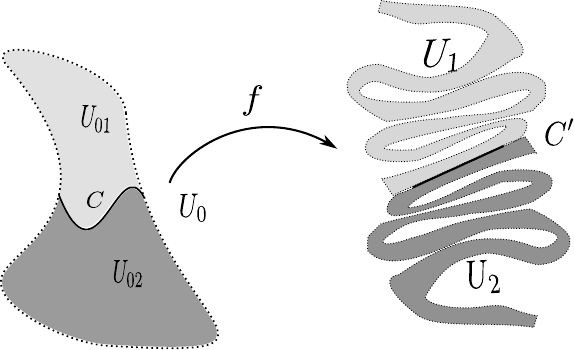}
\caption{Components $U_{01}$ and $U_{02}$ in $U_0$.}
\end{figure}

Given $y \in C'$, denote by $B_{\epsilon}(y)$ a ball in $\T^2$ centered in $y$ and radius $\epsilon$.\\
\newline
{{\it{Claim}} 1:} For every $\epsilon \ll 1$, we have that $B_{\epsilon}(y) \cap \overline{U}_1$ or $B_{\epsilon}(y) \cap \overline{U}_2$ has
infinitely many connected components.\\

Indeed, suppose that for every $\epsilon>0$, $B_{\epsilon}(y)\cap \overline{U}_1$ and $B_{\epsilon}(y)\cap \overline{U}_2$ has
finitely many connected components. Denote by $W^+$ the connected component of $B_{\epsilon}(y)\cap \overline{U}_1$ and
by $W^-$ the connected component of $B_{\epsilon}(y)\cap \overline{U}_2$ which intersect $C'$. Note that, up to subsets of $C'$,
we may suppose that $C'\subseteq B_{\epsilon}(y)$ and that $C'= W^+\cap W^-$.

\begin{figure}[!h]\label{finitascomponentes}
\centering
\includegraphics[scale=0.6]{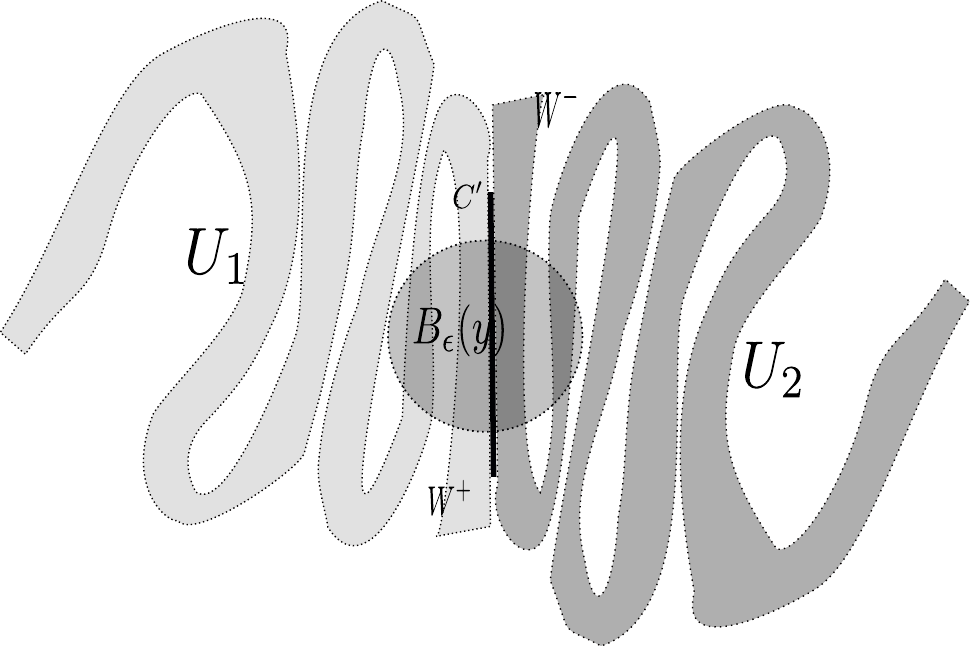}
\caption{connected components.}
\end{figure}

Hence, we can choose $\epsilon_0>0$ such that $W^+\cup W^-$ contain an open set and $B_{\epsilon}(y)$ is contained in $W^+\cup W^-$
for every $0<\epsilon<\epsilon_0$. In particular, $B_{\epsilon}(y)$ is contained in $\overline{U}_1\cup \overline{U}_2$.
Contradicting the fact that $U={\rm{int}}(\overline{U})$ and $U_1, U_2$ are connected components of $U$. This proves of claim 1.

To finish the proof of the lemma, we may suppose, without loss generality, that \linebreak $B_{\epsilon_0}(y)\cap U_1$ has infinitely many connected
components. Then, we know, by continuity of $f$, that for $0<\epsilon<\frac{\epsilon_0}{2}$ there is $\delta >0$ such that
$$d(x,y)<\delta \Rightarrow d(f(x),f(y))<\epsilon, \forall x,y \in \T^2.$$
Now, we consider $x \in C$ such that $y=f(x)$ and a curve $\gamma$ in $B_{\delta}(x)$ that intersect $C$ at $x$ and 
$\gamma(0)\in U_{01},\gamma(1)\in U_{02}$. Then $f(\gamma)$ is a curve such that $f(\gamma)\cap B_{\epsilon_0}(y)$ has infinitely many
components. In particular, there exist $t,s \in [0,1]$ such that $$d(f(\gamma(t)),f(\gamma(s)))\geq \epsilon_0>\epsilon,$$
which is a contradiction, because $f$ is uniformly continuous, the desired result follows.

\begin{figure}[!h]\label{abertos}
\centering
\includegraphics[scale=1.1]{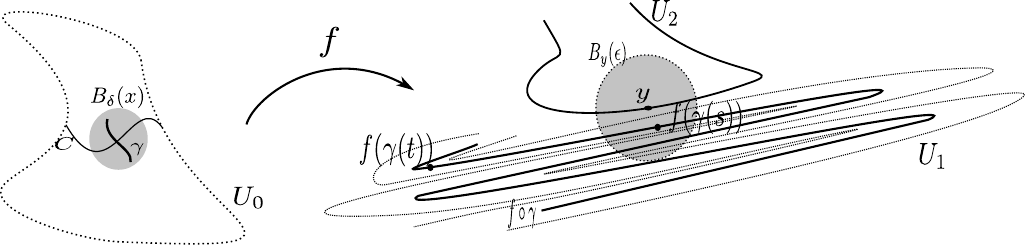}
\caption{Connected components.}
\end{figure}

\end{proof}

The lemma below is important because it shows the existence of critical points that are not generic for $f$.

\begin{cor}\label{cor 3}
Let $U_{01}$ and $U_{02}$ be as in Lemma $\ref{lema 2}$. If $p$ belongs to \linebreak $C=\partial U_{01}\cap \partial U_{02}$ then
$p$ is not a generic critical point.
\end{cor}

\begin{proof}
By item ($b$) of the Lemma \ref{lema 1} and by Remark \ref{rmk 1}, $U$ and $V$ are disjoint $f$-backward invariant open sets satisfying:
\begin{itemize}
 \item $\T^2=\overline{U} \cup V$;
 \item $f(\partial U)=\partial U$ and $f(\partial V)=\partial V$.  
\end{itemize}
Then, $f^{-1}(f(p))$ has empty interior. Otherwise, $f({\rm{int}}(f^{-1}(f(p)))=f(p) \in \partial U$ that is
$f$-forward invariant, contradicting the fact that $f$ is a non-wandering endomorphism. 
Now, we can choose a neighborhood $B$ of $p$ contained in $U_0$ such that $B\backslash \{f^{-1}(f(p))\}$ has at least two
connected components which are contained in $U_{01}$ and $U_{02}$.
By Lemma \ref{lema 2}, it follows that the boundary component of $U_{0i}$ contained in $U_0$ has as image a point, where $U_{0i}$
is a component connected of $f^{-1}(U_i)$ contained in $U_0$. Then, as $\overline{U}_0=\{\overline{U}_{0i}:U_{0i} \subset U_0\}$, we have that
$$f(B)\backslash \{f(p)\}= f(B\backslash \{f^{-1}(f(p))\}) \subset \{f(B\cap\overline{U}_{0i}):U_{0i}\subset U_0\}.$$
In particular,  $f(B\cap U_{01})\subseteq U_1$ and $f(B\cap U_{02})\subseteq U_2$.

\begin{figure}[!h]\label{discpoint}
\centering
\includegraphics[scale=0.8]{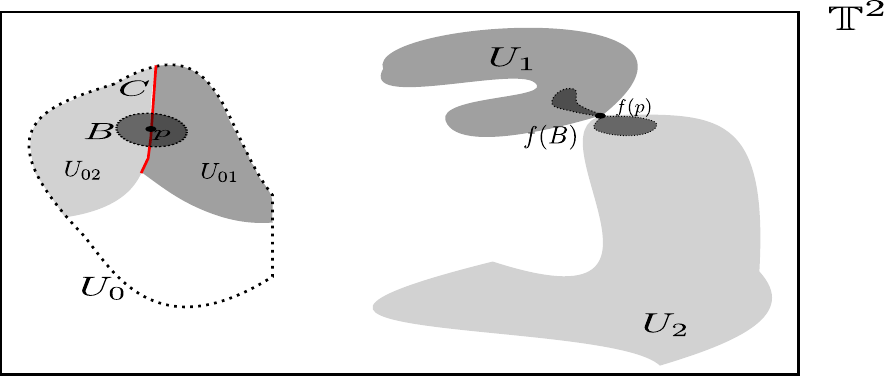}
\caption{$p$ is not generic critical point.}
\end{figure}

Therefore, one has that $p$ is not a generic critical point.

\end{proof}

In the following lemma we will show that $f$ satisfies: for each $U_i$ connected component of $U$ there exists a unique connected
component $U_j$ of $U$ such that $\overline{U}_j=f(\overline{U}_i)$. Hence, we will say that $f$ {\it{preserves the connected
components of}} $U$.

Now, we suppose, in addition to the hypothesis of $f$ be non-wandering endomorphism of degree at least two,
that the critical points of $f$ are generics.

\begin{lemma}\label{lema 3}
$f$ preserves the connected components of $U$. Moreover, every connected component $U_i$ of $U$ is
periodic $(i.e.,\exists n_i\geq 1$ such that $f^{n_i}(\overline{U}_i)=\overline{U}_i$ and $f^{-n_i}(U_i) \subseteq U_i).$
\end{lemma}

\begin{proof}
Suppose that $f(U_i)$ intersect at least two connected components of $U$. Then, by Corollary \ref{cor 3}, it follows that there
exists a non-generic critical point, contradicting that $S_f$ is a generic critical set. Thus, we have that for each
connected component $U_i$ of $U$, $f(U_i)$ must intersect a unique connected component $U_{j_i}$ of $U$.
In particular, since $f(\partial U_i) \subseteq \partial U$, one has $f(\overline{U}_i) \subseteq \overline{U}_{j_i}$.
More precisely, one has that for each connected component $U_i$ of $U$ there exists a unique $U_{j_i}$ such that
$f(\overline{U_i})\subseteq \overline{U}_{j_i}$.

We want to prove that every connected component $U_i$ of $U$ is periodic but before that,
we prove that for each $U_i$ there exists a unique $U_{j}$ such that $f^{-1}(U_i)\subseteq U_{j}$.

Indeed, suppose that $f^{-1}(U_i)$ intersects at least two connected components $U_j$ and $U_k$ of $U$. Then, by we saw above,
we have that $f(\overline{U}_j)\subseteq \overline{U}_i$ and $f(\overline{U}_k)\subseteq \overline{U}_i$.
Since $\Omega(f)=\T^2$, there exist $n_i, n_k \geq 1$ such that $f^{n_j}(\overline{U}_j)\subseteq \overline{U}_j$ and
$f^{n_j}(\overline{U}_j)\subseteq \overline{U}_j$ that imply $f^{n_j-1}(U_i)\subseteq U_j$ and $f^{n_k-1}(U_i)\subseteq U_k$.
Hence, one has $n_j=n_k$ and $U_j=U_k$.

Therefore, for each connected component $U_i$ of $U$ there exist unique $U_{j_i}$ and $U_{k_i}$ such that
$f^{-1}(U_i)\subseteq U_{j_i}$ and $f(\overline{U}_i)\subseteq \overline{U}_{k_i}$ implying that
$f$ preserves the connected components of $U$, $f^{n_i}(\overline{U}_i)=\overline{U}_i$,
and $f^{-n_i}(U_i)\subseteq U_i$.
\end{proof}

\begin{cor}\label{lema 5}
There is a finite number of connected components of $U$.
\end{cor}

\begin{proof}
By definition of $U$ (see equation \ref{def-U}), we can take a connected component $U_0$ of $U$ such that
$\overline{U}=\overline{\cup_{n\geq 0}f^{-n}(U_0)}$. Hence and by Lemma \ref{lema 3}, for each connected
component $U_j$ of $U$ there exists $n_j \geq 1$ and $n_0\geq 1$ such that $f^{n_j}(\overline{U}_j)=\overline{U}_0$
and $f^{n_0}(\overline{U}_0)=\overline{U}_0$. Therefore, $U$ has finitely many connected components.
\end{proof}

\section{Essential sets}\label{section-5}

Now, our goal is to show that $f$ for a non-wandering endomorphism with topological degree at least two and generic critical set
that is not itself transitive, every connected component of $U$ has fundamental group with just one generator in the fundamental
group of the torus. Before to formalize this idea, let us fix some notations. Let $L$ be the linear part of $f$ which is an
invertible matrix in $M_2(\Z)$ and has determinant of modulus at least two. Let $\pi:\R^2 \to \T^2$ be the universal
covering of the torus and let $\tilde{f}:\R^2 \to \R^2$ be a lift of $f$. It is known that
$\tilde{f}(\tilde{x}+v)=L(v)+\tilde{f}(\tilde{x})$ for every $\tilde{x} \in \R^2$ and $v \in \Z$.  

\begin{defi}\label{def-1}
We say that a connected open set $A$ in $\T^2$ is essential if for every connected component $\widetilde{A}$ of $\pi^{-1}(A)$ in $\R^2$,
$$\pi:=\pi|_{\widetilde{A}}:\widetilde{A} \to A$$ is not a homeomorphism. Otherwise, we say that $A$ is inessential.
\end{defi}

The following proposition shows properties of the essential sets.

\begin{prop}\label{prop 1}
Let $W\subseteq \T^2$ be a connected open set. Then the following are equivalent:
\begin{enumerate}
\item[$(i)$] $W$ is essential;
\item[$(ii)$] $W$ contains a loop homotopically non-trivial in $\T^2$;
\item[$(iii)$] there is a non-trivial deck transformation $T_w:\R^2\to \R^2$ such that every connected component of $\pi^{-1}(W)$
is $T_w$-invariant.
\end{enumerate}
Moreover, if $W$ is path connected in $\T^2$, then $i_\ast:\pi_1(W,x)\to \pi_1(\T^2,x)$ is a non-trivial map, where $x\in W$ and 
$i:W\hookrightarrow \T^2$ is the inclusion.
\end{prop}

Heuristically, an essential set is a set that every connected component of its lift has infinite volume.

The following lemma is fundamental in the proof of Main Theorem. That lemma is interesting, because it shows
that every closure of a connected component of $U$ contains a closed curve homotopically non-trivial in $\T^2$.

\begin{lemma}\label{lema 4}
Let $U_j$ be any connected component of $U$. Then $\overline{U}_j$ contains a closed curve homotopically non-trivial in $\T^2$.
\end{lemma}

\begin{proof}
By Corollary \ref{lema 5}, we can suppose $\overline{U}_j=f^j(\overline{U}_0)$ and $U_{n_0}=U_0$.
If $U_0$ is essential there is nothing to prove. Now, suppose that $U_0$ is inessential.
Let $\widetilde{U}_0 \subset \R^2$ be a connected component of $\pi^{-1}(U_0)$, then, $\pi: \widetilde{U}_0\to U_0$ is injective.
Consider $w \in \Z^2 \backslash L(\Z^2)$, such $w$ exists because $|\det(L)|\geq 2$. We 
denote by $W'$ the interior of the set $\pi(\tilde{f}^{-1}(w+\tilde{f}(\widetilde{U}_0)))$ that is not empty, because $f(U_0)$
has interior non-empty. Then
\begin{center}
$f(W')=f\circ \pi(\tilde{f}^{-1}(w+\tilde{f}(\widetilde{U}_0)))
=\pi(w+\tilde{f}(\widetilde{U}_0))=f(U_0)$. 
\end{center}
But as $W'$ is a open set and $f(W') \subset \overline{U}_1$, and so $W' \subset U$. Then, 
$$f^n(W')\cap W'\not=\emptyset \Longleftrightarrow n=kn_0, \ \ {\rm{for \ \ some}} \ \ k\geq 1.$$ 
In particular, $W'$ must intersect to $U_0$. Hence $W'$ is contained in $U_0$.

Since $W'$ is contained in $U_0$, we have that $\tilde{f}^{-1}(w+\tilde{f}(\widetilde{U}_0))$ is contained in $\widetilde{U}_0$.
Thus, $\tilde{f}(\widetilde{U}_0)$ contains $w+\tilde{f}(\widetilde{U}_0)$ and $\tilde{f}(\widetilde{U}_0)$.
Hence, there exist $\tilde{x}$ and $\tilde{y}$ in $\widetilde{U}_0$ such that
$\tilde{f}(\tilde{y})=w+\tilde{f}(\tilde{x})$, and so taking a curve $\tilde{\gamma}$ in $\widetilde{U}_0$
joining $\tilde{x}$ to $\tilde{y}$, one has $\tilde{f}(\tilde{\gamma})$ is a curve joining $\tilde{f}(\tilde{x})$ to $\tilde{f}(\tilde{y})$.
In particular, $\gamma:=\pi\circ \tilde{\gamma}$ is a curve such that $\gamma_f:=f\circ \gamma$ is a closed curve whose
homology class is $w$. Therefore, $f^{j-1}\circ \gamma_f$ is a closed curve in $\overline{U}_j$ whose homology class is $L^{j-1}(w)$.
\end{proof}

Next lemma is important because it shows that the closure of the connected components of $U$ and $V$ obtained in Lemma \ref{lema 1}
are essential sets. 

\begin{lemma}\label{lema 8}
If $U_j$ is a connected component of $U$ such that $f^n(\overline{U}_j)=\overline{U}_j$ for some $n\geq 1$. Then, $U_j$ is essential.
\end{lemma}

\begin{proof}
Suppose, without loss of generality, $j=0$. Let $\widetilde{U}_0$ be a connected component of $\pi^{-1}(U_0)$ in $\R^2$.
Suppose that $U_0$ is an inessential set. Since the degree of $f$ is at least two and $f(\partial U_0)\subseteq \partial U$,
one has that $\tilde{f}(\widetilde{U}_0)$ contains at least two connected components of $\pi^{-1}(U_0)$ and
$\tilde{f}(\partial \widetilde{U}_0)\subseteq \partial \pi^{-1}(U_0)$.
Then, there exist at least two connected components of $\pi^{-1}(U_0)$, suppose, without loss of generality, that
$\widetilde{U_0}$ and $\widetilde{U_0}+v$ for some $v \in \Z^2$ are contained in $\tilde{f}(\widetilde{U}_0)$ and that
the component components $\widetilde{U}_{00}$ and $\widetilde{U}_{0v}$ of $\tilde{f}^{-1}(\widetilde{U}_0)$
and $\tilde{f}^{-1}(\widetilde{U}_w)$, respectively, contained in $\widetilde{U}_0$ so that
$\widetilde{C}=\partial \widetilde{U}_{00}\cap \partial\widetilde{U}_{0w}$ is a nonempty set in $\widetilde{U}_0$.

\begin{figure}[!h]
\centering
\includegraphics[scale=0.7]{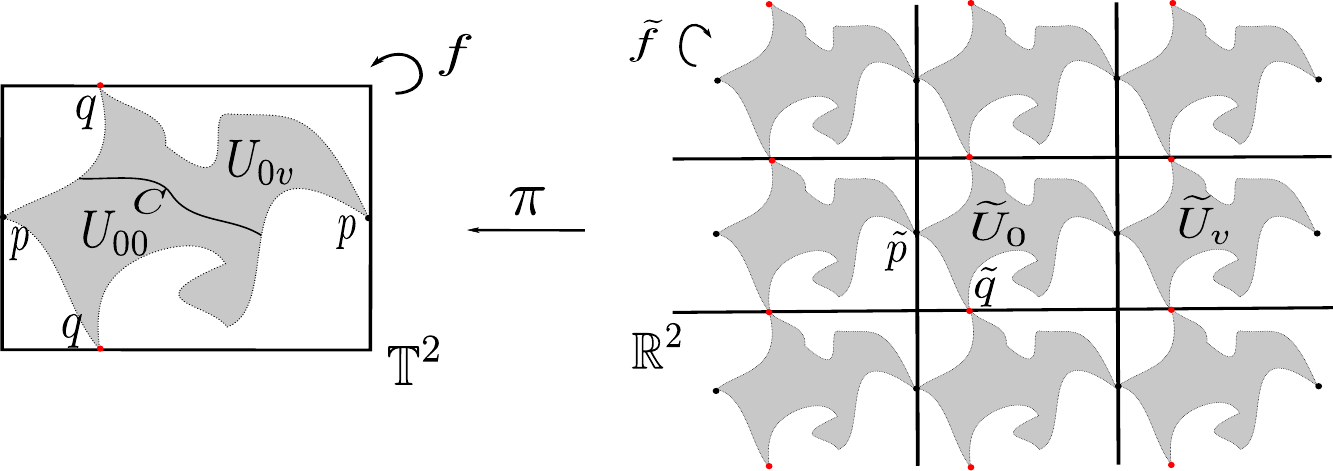}
\caption{The components $\widetilde{U}_0$ and $\widetilde{U}_v$.}
\end{figure}

Then, from the proof of Lemma \ref{lema 2}, $f(\pi(\widetilde{C}))$ is a point and ,by the proof of the Corollary \ref{cor 3},
there exists $p \in \pi(\widetilde{C})$ so that $p$ is not a generic critical point. Contradicting that $S_f$
is a generic critical set.

\end{proof}

The lemma below shows what happens when two essential sets are linearly independent.

\begin{lemma}\label{lema 6}
Suppose that $\gamma$ and $\sigma$ are loops in $\T^2$ such that $[\gamma]$ and $[\sigma]$ are linearly independent in $\Z^2$.
Then $\gamma$ and $\sigma$ intersect. 
\end{lemma}

\begin{proof}
See Lemma $4.2$ in \cite{Andersson}.
\end{proof}

The lemma below shows the existence of integer eigenvalues of $L$.

\begin{lemma}\label{lema 7}
The eigenvalues of $L$ are integers.
\end{lemma}

\begin{proof}
By Lemma \ref{lema 8}, the connected components $U_j$ and $V_i$ of $U$ and $V$ are essentials. We consider two
loops $\gamma$ and $\sigma$ in $U_j$ and $V_i$ such that $[\gamma]$ and $[\sigma]$ are different to zero in $\Z^2$.
As $U_j\cap V_i=\emptyset$, it follows, by Lemma \ref{lema 6}, that $[\gamma]$ and $[\sigma]$ are linearly dependent in $\Z^2$.
analogously, as $\overline{U}_{j+1}\cap V_i=\emptyset$ and $f \circ \gamma$ is loop in $\overline{U}_{j+1}$, we have that
$[f \circ \gamma]=L[\gamma]$ and $[\sigma]$ are linearly dependent in $\Z^2$, in particular 
$L[\gamma]$ and $[\gamma]$ are linearly dependent in $\Z^2$. Therefore, there exits $k \in \Z\backslash\{0\}$
such that $L[\gamma]=k[\gamma]$. This proves the lemma.
\end{proof}

The lemma below is fundamental. It shows that all connected components of $U$ and $V$ are essential.

\section{The proof of Main Theorem}\label{section-6}

Let $f:\T^2 \to \T^2$ be a non-wandering endomorphism with generic critical set and degree at least two which is not transitive.
Then we know from Lemma \ref{lema 1} that there exist $U$ and $V$ in $\T^2$ $f$-backward invariant regular open sets such that $\overline{U}$
and $\overline{V}$ are $f$-forward invariant sets. Since all critical points are generics, from Lemma \ref{lema 8} and Corollary \ref{lema 5}
follow that all connected component of $U$ and $V$ are essential and that $\overline{U}_0$ is periodic. Let
$\overline{U}_0, f(\overline{U}_0),\cdots,f^{n-1}(\overline{U}_0)$ be all connected components of $\overline{U}$ with
$\overline{U}_0=f^n(\overline{U}_0)$. Then, consider two connected components $\widetilde{U}_0$ and $\widetilde{V}_0$ of $\pi^{-1}(U_0)$
and $\pi^{-1}(V_0)$, respectively, and choose $\tilde{f}:\R^2\to \R^2$ a lift of $f$ such that
$\tilde{f}^n(\widetilde{U}_0)\subseteq {\bf{cl}}(\widetilde{U}_0)$.

\begin{figure}[!h]
\centering
\includegraphics[scale=0.8]{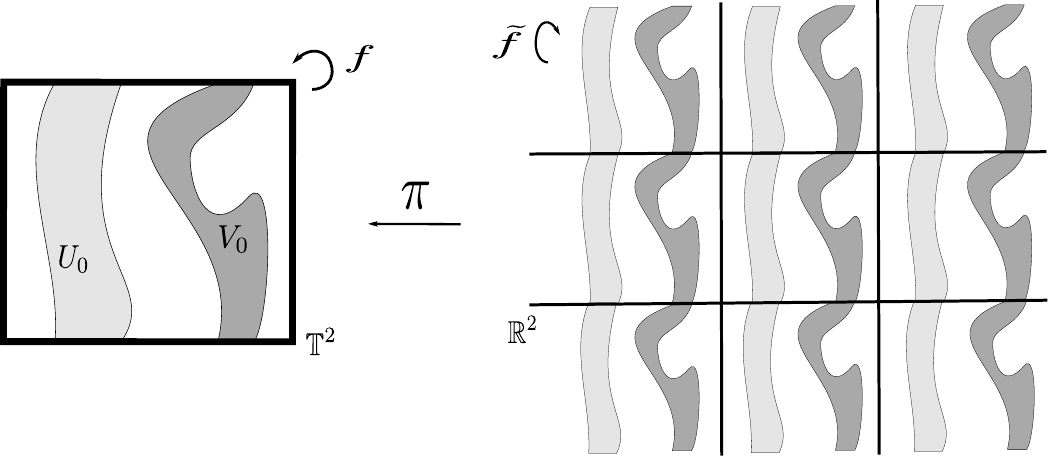}
\caption{The sets $U_0$ and $V_0$.}
\end{figure}

Let us now prove that $L$ has a real eigenvalue of modulus one. First, note that as $\overline{U}_0$ and $V_0$ are disjoints,
Lemma \ref{lema 7} implies that $L$ has integer eigenvalues $l$ and $k$. Let $w$ and $u$ be the eigenvectors of $L$ associated to $l$
and $k$, respectively, are in $\Z^2$. Suppose, without loss generality, that $w$ and $u=e_2$. That is, as $\widetilde{U}_0$ is
$T_u$-invariant, we have that $\widetilde{U}_0$ is a "vertical" component of $\pi^{-1}(U_0)$.

To finish, suppose that $|k|\geq 2$. Then, consider in $\R^2$ a curve $\widetilde{\gamma}$ which $\widetilde{\gamma}(0) \in \widetilde{U}_0$ and
$\widetilde{\gamma}(1)=\widetilde{\gamma}(0)+e_1$. Thus, $\widetilde{f}^{n}\circ\widetilde{\gamma}$ is a curve with
$\widetilde{f}^{n}\circ\widetilde{\gamma}(0)\in \widetilde{U}_0$ and \linebreak
$\widetilde{f}^{n}\circ\widetilde{\gamma}(1)=\widetilde{f}^{n}\circ\widetilde{\gamma}(0)+L^{n}(e_1)$. However, there exist $a$ and $b$ in $\Z$
with $b$ different from zero such that $e_1=ae_2+bw$. Hence, we have $L(e_1)=ake_2+blw$ and, in particular,

$$\tilde{f}(\widetilde{U}_0+e_1)\subseteq {\bf{cl}}(\widetilde{U}_0+L(e_1)).$$

\begin{figure}[!h]
\centering
\includegraphics[scale=0.8]{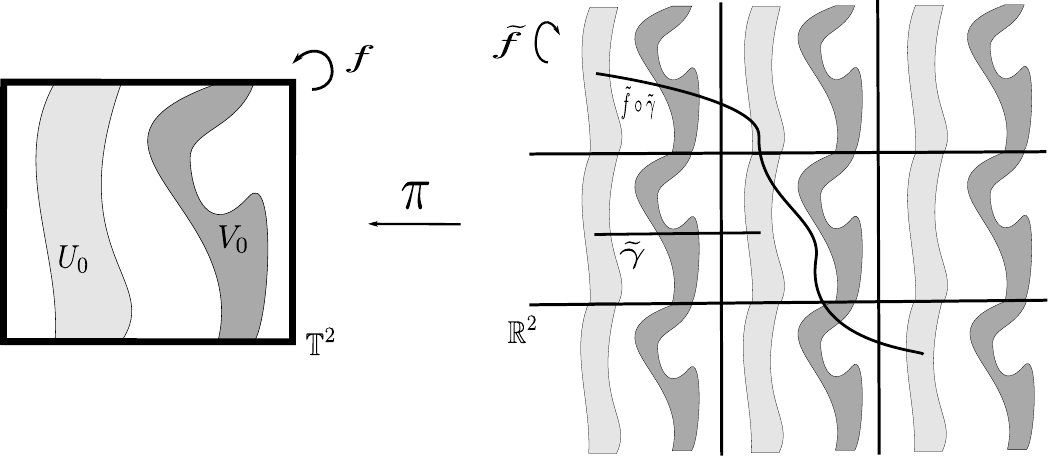}
\caption{ The curves $\widetilde{\gamma}$ and $\widetilde{f}\circ\widetilde{\gamma}$.}
\end{figure}

Then, if $|l| \geq 2$, we have that the first coordinate of $L(e_1)$ has modulus at least two. Hence there is $c \in \Z$ such that
$\widetilde{U}_r+ce_1$ is between $\widetilde{U}_0$ and $\widetilde{U}_0+L(e_1)$, and so, we have a contradiction because $U_r's$
are disjoint and $f^n$-backward invariant. Hence there is not a set $W$ in between $\widetilde{U}_0$ and $\widetilde{U}_0+e_1$ such that
$\tilde{f}^n(W)=\widetilde{U}_0+ce_1$.

Therefore,  $|l|=1$. And so, $L$ is not an Anosov endomorphism, contradiction. This proves the Main Theorem. 

\section{Examples}

Consider a map $f$ on $\mathbb{S}^1$ itself of form:
\begin{figure}[!h]
\centering
\includegraphics[scale=0.8]{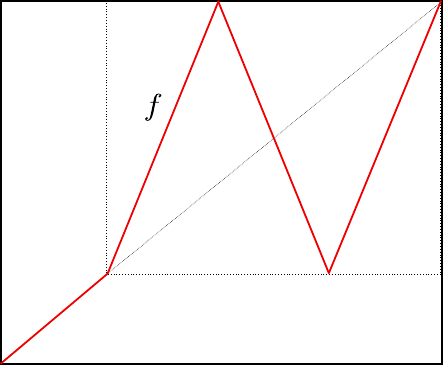}
\caption{The graph of $f$.}
\end{figure}\\
$\noindent$ such that $f$ is not transitive map, but is volume preserving. Let $g: \mathbb{S}^1 \to \mathbb{S}^1$ be any volume preserving
degree 2 map and let $H:\T^2 \to \T^2$ be any volume preserving endomorphism without critical points homotopic to $(x,y) \mapsto (2x, y)$.
Then $F: \T^2 \to \T^2$ given by $$F(x,y)= H(f(x),g(y))$$ is a volume preserving endomorphism with generic critical points and homotopic to
$(x,y) \mapsto (2x, 2y)$. Therefore, by the Main Theorem, $F$ is transitive. More general, given endomorphisms $f$ and $g$ on $\mathbb{S}^1$
itself which $f$ has critical points and $g$ is an expanding such that $f\times g: \T^2 \to \T^2,\, f\times g (x,y)=(f(x),g(y))$, is a volume
preserving endomorphism, then for every $H : \T^2 \to \T^2$ a volume preserving covering map such that $F=H \circ (f\times  g)$ is homotopic
to Anosov endomorphism degree at least two, we have $F$ is transitive.
 
\section*{Acknowledgments}
The author\footnote{The author has been supported by CAPES.} would like to express his gratitude to Professor Enrique Pujals and to Rafael Potrie,
as well as to Professor Cristina Lizana for their helpful comments and appropriate advices. He would also like to thank to Martin Andersson and
Andres Koropecki to the discussions about this work.


\bibliographystyle{alpha}

\bibliography{biblio-1}

\nocite{*}

\end{document}